\documentclass[12pt]{amsart}       
\usepackage{txfonts}
\usepackage{amssymb}
\usepackage{eucal}
\usepackage{graphicx}
\usepackage{amsmath}
\usepackage{amscd}
\usepackage[all]{xy}           
\usepackage{tikz}
\usepackage{amsfonts,latexsym}
\usepackage{xspace}
\usepackage{epsfig}
\usepackage{float}
\usepackage{axodraw}
\usepackage{color}
\usepackage{fancybox}
\usepackage{colordvi}
\usepackage{multicol}
\usepackage{colordvi}
\usepackage[active]{srcltx} 

\newtheorem{theorem}{Theorem}[section]
\newtheorem{prop}[theorem]{Proposition}
\theoremstyle{definition}
\newtheorem{defn}[theorem]{Definition}
\newtheorem{lemma}[theorem]{Lemma}

\newtheorem{prop-def}{Proposition-Definition}[section]
\newtheorem{coro-def}{Corollary-Definition}[section]

\newtheorem{remark}[theorem]{Remark}

\newcommand{\nc}{\newcommand}
\nc{\tred}[1]{\textcolor{red}{#1}}
\nc{\tblue}[1]{\textcolor{blue}{#1}}
\nc{\tgreen}[1]{\textcolor{green}{#1}}
\nc{\tpurple}[1]{\textcolor{purple}{#1}}
\nc{\btred}[1]{\textcolor{red}{\bf #1}}
\nc{\btblue}[1]{\textcolor{blue}{\bf #1}}
\nc{\btgreen}[1]{\textcolor{green}{\bf #1}}
\nc{\btpurple}[1]{\textcolor{purple}{\bf #1}}
\nc{\NN}{{\mathbb N}}
\nc{\ncsha}{{\mbox{\cyr X}^{\mathrm NC}}} \nc{\ncshao}{{\mbox{\cyr
X}^{\mathrm NC}_0}}

\renewcommand{\frak}{\mathfrak}

\newcommand{\efootnote}[1]{}

\renewcommand{\textbf}[1]{}

\newcommand{\delete}[1]{}

\nc{\mlabel}[1]{\label{#1}}  
\nc{\mcite}[1]{\cite{#1}}  
\nc{\mref}[1]{\ref{#1}}  
\nc{\mbibitem}[1]{\bibitem{#1}} 

\delete{
\nc{\mlabel}[1]{\label{#1}  
{\hfill \hspace{1cm}{\small\tt{{\ }\hfill(#1)}}}}
\nc{\mcite}[1]{\cite{#1}{\small{\tt{{\ }(#1)}}}}  
\nc{\mref}[1]{\ref{#1}{{\tt{{\ }(#1)}}}}  
\nc{\mbibitem}[1]{\bibitem[\bf #1]{#1}} 
}

\nc{\opa}{\ast} \nc{\opb}{\odot} \nc{\op}{\bullet} \nc{\pa}{\frakL}
\nc{\arr}{\rightarrow} \nc{\lu}[1]{(#1)} \nc{\mult}{\mrm{mult}}
\nc{\diff}{\mathfrak{Diff}}
\nc{\opc}{\sharp}\nc{\opd}{\natural}
\nc{\ope}{\circ}
\nc{\dpt}{\mathrm{d}}
\nc{\diam}{alternating\xspace}
\nc{\Diam}{Alternating\xspace}
\nc{\cdiam}{alternating\xspace}
\nc{\Cdiam}{Alternating\xspace}
\nc{\AW}{\mathcal{A}}
\nc{\rba}{Rota-Baxter algebra\xspace}

\nc{\ari}{\mathrm{ar}}

\nc{\lef}{\mathrm{lef}}

\nc{\Sh}{\mathrm{ST}}

\nc{\Cr}{\mathrm{Cr}}

\nc{\st}{{Schr\"oder tree}\xspace}
\nc{\sts}{{Schr\"oder trees}\xspace}

\nc{\vertset}{\Omega} 

\nc{\assop}{\quad \begin{picture}(5,5)(0,0)
\line(-1,1){10}
\put(-2.2,-2.2){$\bullet$}
\line(0,-1){10}\line(1,1){10}
\end{picture} \quad \smallskip}

\nc{\operator}{\begin{picture}(5,5)(0,0)
\line(0,-1){6}
\put(-2.6,-1.8){$\bullet$}
\line(0,1){9}
\end{picture}}

\nc{\idx}{\begin{picture}(6,6)(-3,-3)
\put(0,0){\line(0,1){6}}
\put(0,0){\line(0,-1){6}}
 \end{picture}}

\nc{\pb}{{\mathrm{pb}}}
\nc{\Lf}{{\mathrm{Lf}}}

\nc{\lft}{{left tree}\xspace}
\nc{\lfts}{{left trees}\xspace}

\nc{\fat}{{fundamental averaging tree}\xspace}

\nc{\fats}{{fundamental averaging trees}\xspace}
\nc{\avt}{\mathrm{Avt}}

\nc{\rass}{{\mathit{RAss}}}

\nc{\aass}{{\mathit{AAss}}}

\nc{\vin}{{\mathrm Vin}}    
\nc{\lin}{{\mathrm Lin}}    
\nc{\inv}{\mathrm{I}n}
\nc{\gensp}{V} 
\nc{\genbas}{\mathcal{V}} 
\nc{\bvp}{V_P}     
\nc{\gop}{{\,\omega\,}}     

\nc{\bin}[2]{ (_{\stackrel{\scs{#1}}{\scs{#2}}})}  
\nc{\binc}[2]{ \left (\!\! \begin{array}{c} \scs{#1}\\
    \scs{#2} \end{array}\!\! \right )}  
\nc{\bincc}[2]{  \left ( {\scs{#1} \atop
    \vspace{-1cm}\scs{#2}} \right )}  
\nc{\bs}{\bar{S}} \nc{\cosum}{\sqsubset} \nc{\la}{\longrightarrow}
\nc{\rar}{\rightarrow} \nc{\dar}{\downarrow} \nc{\dprod}{**}
\nc{\dap}[1]{\downarrow \rlap{$\scriptstyle{#1}$}}
\nc{\md}{\mathrm{dth}} \nc{\uap}[1]{\uparrow
\rlap{$\scriptstyle{#1}$}} \nc{\defeq}{\stackrel{\rm def}{=}}
\nc{\disp}[1]{\displaystyle{#1}} \nc{\dotcup}{\
\displaystyle{\bigcup^\bullet}\ } \nc{\gzeta}{\bar{\zeta}}
\nc{\hcm}{\ \hat{,}\ } \nc{\hts}{\hat{\otimes}}
\nc{\barot}{{\otimes}} \nc{\free}[1]{\bar{#1}}
\nc{\uni}[1]{\tilde{#1}} \nc{\hcirc}{\hat{\circ}} \nc{\lleft}{[}
\nc{\lright}{]} \nc{\lc}{\lfloor} \nc{\rc}{\rfloor}
\nc{\curlyl}{\left \{ \begin{array}{c} {} \\ {} \end{array}
    \right .  \!\!\!\!\!\!\!}
\nc{\curlyr}{ \!\!\!\!\!\!\!
    \left . \begin{array}{c} {} \\ {} \end{array}
    \right \} }
\nc{\longmid}{\left | \begin{array}{c} {} \\ {} \end{array}
    \right . \!\!\!\!\!\!\!}
\nc{\onetree}{\bullet} \nc{\ora}[1]{\stackrel{#1}{\rar}}
\nc{\ola}[1]{\stackrel{#1}{\la}}
\nc{\ot}{\otimes} \nc{\mot}{{{\boxtimes\,}}}
\nc{\otm}{\overline{\boxtimes}} \nc{\sprod}{\bullet}
\nc{\scs}[1]{\scriptstyle{#1}} \nc{\mrm}[1]{{\rm #1}}
\nc{\margin}[1]{\marginpar{\rm #1}}   
\nc{\dirlim}{\displaystyle{\lim_{\longrightarrow}}\,}
\nc{\invlim}{\displaystyle{\lim_{\longleftarrow}}\,}
\nc{\mvp}{\vspace{0.3cm}} \nc{\tk}{^{(k)}} \nc{\tp}{^\prime}
\nc{\ttp}{^{\prime\prime}} \nc{\svp}{\vspace{2cm}}
\nc{\vp}{\vspace{8cm}} \nc{\proofbegin}{\noindent{\bf Proof: }}
\nc{\proofend}{$\blacksquare$ \vspace{0.3cm}}
\nc{\modg}[1]{\!<\!\!{#1}\!\!>}
\nc{\intg}[1]{F_C(#1)} \nc{\lmodg}{\!
<\!\!} \nc{\rmodg}{\!\!>\!}
\nc{\cpi}{\widehat{\Pi}}
\nc{\sha}{{\mbox{\cyr X}}}  
\nc{\shap}{{\mbox{\cyrs X}}} 
\nc{\shpr}{\diamond}    
\nc{\shp}{\ast} \nc{\shplus}{\shpr^+}
\nc{\shprc}{\shpr_c}    
\nc{\msh}{\ast} \nc{\zprod}{m_0} \nc{\oprod}{m_1}
\nc{\vep}{\varepsilon} \nc{\labs}{\mid\!} \nc{\rabs}{\!\mid}
\nc{\sqmon}[1]{\langle #1\rangle}

\nc{\mmbox}[1]{\mbox{\ #1\ }} \nc{\dep}{\mrm{dep}} \nc{\fp}{\mrm{FP}}
\nc{\rchar}{\mrm{char}} \nc{\End}{\mrm{End}} \nc{\Fil}{\mrm{Fil}}
\nc{\Mor}{Mor\xspace} \nc{\gmzvs}{gMZV\xspace}
\nc{\gmzv}{gMZV\xspace} \nc{\mzv}{MZV\xspace}
\nc{\mzvs}{MZVs\xspace} \nc{\Hom}{\mrm{Hom}} \nc{\id}{\mrm{id}}
\nc{\im}{\mrm{im}} \nc{\incl}{\mrm{incl}} \nc{\map}{\mrm{Map}}
\nc{\mchar}{\rm char} \nc{\nz}{\rm NZ} \nc{\supp}{\mathrm Supp}

\nc{\Alg}{\mathbf{Alg}} \nc{\Bax}{\mathbf{Bax}} \nc{\bff}{\mathbf f}
\nc{\bfk}{{\bf k}} \nc{\bfone}{{\bf 1}} \nc{\bfx}{\mathbf x}
\nc{\bfy}{\mathbf y}
\nc{\base}[1]{\bfone^{\otimes ({#1}+1)}} 
\nc{\Cat}{\mathbf{Cat}}

\nc{\detail}{\marginpar{\bf More detail}
    \noindent{\bf Need more detail!}
    \svp}
\nc{\Int}{\mathbf{Int}} \nc{\Mon}{\mathbf{Mon}}
\nc{\rbtm}{{shuffle }} \nc{\rbto}{{Rota-Baxter }}
\nc{\remarks}{\noindent{\bf Remarks: }} \nc{\Rings}{\mathbf{Rings}}
\nc{\Sets}{\mathbf{Sets}} \nc{\wtot}{\widetilde{\odot}}
\nc{\wast}{\widetilde{\ast}} \nc{\bodot}{\bar{\odot}}
\nc{\bast}{\bar{\ast}} \nc{\hodot}[1]{\odot^{#1}}
\nc{\hast}[1]{\ast^{#1}} \nc{\mal}{\mathcal{O}}
\nc{\tet}{\tilde{\ast}} \nc{\teot}{\tilde{\odot}}
\nc{\oex}{\overline{x}} \nc{\oey}{\overline{y}}
\nc{\oez}{\overline{z}} \nc{\oef}{\overline{f}}
\nc{\oea}{\overline{a}} \nc{\oeb}{\overline{b}}
\nc{\weast}[1]{\widetilde{\ast}^{#1}}
\nc{\weodot}[1]{\widetilde{\odot}^{#1}} \nc{\hstar}[1]{\star^{#1}}
\nc{\lae}{\langle} \nc{\rae}{\rangle}
\nc{\lf}{\lfloor}
\nc{\rf}{\rfloor}

\def\ta1{{\scalebox{0.25}{ 
\begin{picture}(12,12)(38,-38)
\SetWidth{0.5} \SetColor{Black} \Vertex(45,-33){5.66}
\end{picture}}}}

\def\tb2{{\scalebox{0.25}{ 
\begin{picture}(12,42)(38,-38)
\SetWidth{0.5} \SetColor{Black} \Vertex(45,-3){5.66}
\SetWidth{1.0} \Line(45,-3)(45,-33) \SetWidth{0.5}
\Vertex(45,-33){5.66}
\end{picture}}}}

\def\tc3{{\scalebox{0.25}{ 
\begin{picture}(12,72)(38,-38)
\SetWidth{0.5} \SetColor{Black} \Vertex(45,27){5.66}
\SetWidth{1.0} \Line(45,27)(45,-3) \SetWidth{0.5}
\Vertex(45,-33){5.66} \SetWidth{1.0} \Line(45,-3)(45,-33)
\SetWidth{0.5} \Vertex(45,-3){5.66}
\end{picture}}}}

\def\td31{{\scalebox{0.25}{ 
\begin{picture}(42,42)(23,-38)
\SetWidth{0.5} \SetColor{Black} \Vertex(45,-3){5.66}
\Vertex(30,-33){5.66} \Vertex(60,-33){5.66} \SetWidth{1.0}
\Line(45,-3)(30,-33) \Line(60,-33)(45,-3)
\end{picture}}}}

\def\te4{{\scalebox{0.25}{ 
\begin{picture}(12,102)(38,-8)
\SetWidth{0.5} \SetColor{Black} \Vertex(45,57){5.66}
\Vertex(45,-3){5.66} \Vertex(45,27){5.66} \Vertex(45,87){5.66}
\SetWidth{1.0} \Line(45,57)(45,27) \Line(45,-3)(45,27)
\Line(45,57)(45,87)
\end{picture}}}}

\def\tf41{{\scalebox{0.25}{ 
\begin{picture}(42,72)(38,-8)
\SetWidth{0.5} \SetColor{Black} \Vertex(45,27){5.66}
\Vertex(45,-3){5.66} \SetWidth{1.0} \Line(45,27)(45,-3)
\SetWidth{0.5} \Vertex(60,57){5.66} \SetWidth{1.0}
\Line(45,27)(60,57) \SetWidth{0.5} \Vertex(75,27){5.66}
\SetWidth{1.0} \Line(75,27)(60,57)
\end{picture}}}}

\def\tg42{{\scalebox{0.25}{ 
\begin{picture}(42,72)(8,-8)
\SetWidth{0.5} \SetColor{Black} \Vertex(45,27){5.66}
\Vertex(45,-3){5.66} \SetWidth{1.0} \Line(45,27)(45,-3)
\SetWidth{0.5} \Vertex(15,27){5.66} \Vertex(30,57){5.66}
\SetWidth{1.0} \Line(15,27)(30,57) \Line(45,27)(30,57)
\end{picture}}}}

\def\th43{{\scalebox{0.25}{ 
\begin{picture}(42,42)(8,-8)
\SetWidth{0.5} \SetColor{Black} \Vertex(45,-3){5.66}
\Vertex(15,-3){5.66} \Vertex(30,27){5.66} \SetWidth{1.0}
\Line(15,-3)(30,27) \Line(45,-3)(30,27) \Line(30,27)(30,-3)
\SetWidth{0.5} \Vertex(30,-3){5.66}
\end{picture}}}}

\def\thII43{{\scalebox{0.25}{ 
\begin{picture}(72,57) (68,-128)
    \SetWidth{0.5}
    \SetColor{Black}
    \Vertex(105,-78){5.66}
    \SetWidth{1.5}
    \Line(105,-78)(75,-123)
    \Line(105,-78)(105,-123)
    \Line(105,-78)(135,-123)
    \SetWidth{0.5}
    \Vertex(75,-123){5.66}
    \Vertex(105,-123){5.66}
    \Vertex(135,-123){5.66}
  \end{picture}
  }}}

\def\thj44{{\scalebox{0.25}{ 
\begin{picture}(42,72)(8,-8)
\SetWidth{0.5} \SetColor{Black} \Vertex(30,57){5.66}
\SetWidth{1.0} \Line(30,57)(30,27) \SetWidth{0.5}
\Vertex(30,27){5.66} \SetWidth{1.0} \Line(45,-3)(30,27)
\SetWidth{0.5} \Vertex(45,-3){5.66} \Vertex(15,-3){5.66}
\SetWidth{1.0} \Line(15,-3)(30,27)
\end{picture}}}}

\def\ti5{{\scalebox{0.25}{ 
\begin{picture}(12,132)(23,-8)
\SetWidth{0.5} \SetColor{Black} \Vertex(30,117){5.66}
\SetWidth{1.0} \Line(30,117)(30,87) \SetWidth{0.5}
\Vertex(30,87){5.66} \Vertex(30,57){5.66} \Vertex(30,27){5.66}
\Vertex(30,-3){5.66} \SetWidth{1.0} \Line(30,-3)(30,27)
\Line(30,27)(30,57) \Line(30,87)(30,57)
\end{picture}}}}

\def\tj51{{\scalebox{0.25}{ 
\begin{picture}(42,102)(53,-38)
\SetWidth{0.5} \SetColor{Black} \Vertex(61,27){4.24}
\SetWidth{1.0} \Line(75,57)(90,27) \Line(60,27)(75,57)
\SetWidth{0.5} \Vertex(90,-3){5.66} \Vertex(60,27){5.66}
\Vertex(75,57){5.66} \Vertex(90,-33){5.66} \SetWidth{1.0}
\Line(90,-33)(90,-3) \Line(90,-3)(90,27) \SetWidth{0.5}
\Vertex(90,27){5.66}
\end{picture}}}}

\def\tk52{{\scalebox{0.25}{ 
\begin{picture}(42,102)(23,-8)
\SetWidth{0.5} \SetColor{Black} \Vertex(60,57){5.66}
\Vertex(45,87){5.66} \SetWidth{1.0} \Line(45,87)(60,57)
\SetWidth{0.5} \Vertex(30,57){5.66} \SetWidth{1.0}
\Line(30,57)(45,87) \SetWidth{0.5} \Vertex(30,-3){5.66}
\SetWidth{1.0} \Line(30,-3)(30,27) \SetWidth{0.5}
\Vertex(30,27){5.66} \SetWidth{1.0} \Line(30,57)(30,27)
\end{picture}}}}

\def\tl53{{\scalebox{0.25}{ 
\begin{picture}(42,102)(8,-8)
\SetWidth{0.5} \SetColor{Black} \Vertex(30,57){5.66}
\Vertex(30,27){5.66} \SetWidth{1.0} \Line(30,57)(30,27)
\SetWidth{0.5} \Vertex(30,87){5.66} \SetWidth{1.0}
\Line(30,27)(45,-3) \SetWidth{0.5} \Vertex(15,-3){5.66}
\SetWidth{1.0} \Line(15,-3)(30,27) \Line(30,57)(30,87)
\SetWidth{0.5} \Vertex(45,-3){5.66}
\end{picture}}}}

\def\tm54{{\scalebox{0.25}{ 
\begin{picture}(42,72)(8,-38)
\SetWidth{0.5} \SetColor{Black} \Vertex(30,-3){5.66}
\SetWidth{1.0} \Line(30,27)(30,-3) \Line(30,-3)(45,-33)
\SetWidth{0.5} \Vertex(15,-33){5.66} \SetWidth{1.0}
\Line(15,-33)(30,-3) \SetWidth{0.5} \Vertex(45,-33){5.66}
\SetWidth{1.0} \Line(30,-33)(30,-3) \SetWidth{0.5}
\Vertex(30,-33){5.66} \Vertex(30,27){5.66}
\end{picture}}}}

\def\tn55{{\scalebox{0.25}{ 
\begin{picture}(42,72)(8,-38)
\SetWidth{0.5} \SetColor{Black} \Vertex(15,-33){5.66}
\Vertex(45,-33){5.66} \Vertex(30,27){5.66} \SetWidth{1.0}
\Line(45,-33)(45,-3) \SetWidth{0.5} \Vertex(45,-3){5.66}
\Vertex(15,-3){5.66} \SetWidth{1.0} \Line(30,27)(45,-3)
\Line(15,-3)(30,27) \Line(15,-3)(15,-33)
\end{picture}}}}

\nc{\QQ}{{\mathbb Q}}
\nc{\RR}{{\mathbb R}} \nc{\ZZ}{{\mathbb Z}}

\nc{\cala}{{\mathcal A}} \nc{\calb}{{\mathcal B}}
\nc{\calc}{{\mathcal C}}
\nc{\cald}{{\mathcal D}} \nc{\cale}{{\mathcal E}}
\nc{\calf}{{\mathcal F}} \nc{\calg}{{\mathcal G}}
\nc{\calh}{{\mathcal H}} \nc{\cali}{{\mathcal I}}
\nc{\call}{{\mathcal L}} \nc{\calm}{{\mathcal M}}
\nc{\caln}{{\mathcal N}} \nc{\calo}{{\mathcal O}}
\nc{\calp}{{\mathcal P}} \nc{\calr}{{\mathcal R}}
\nc{\cals}{{\mathcal S}} \nc{\calt}{{\mathcal T}}
\nc{\calu}{{\mathcal U}} \nc{\calw}{{\mathcal W}} \nc{\calk}{{\mathcal K}}
\nc{\calx}{{\mathcal X}} \nc{\CA}{\mathcal{A}}

\nc{\fraka}{{\mathfrak a}} \nc{\frakA}{{\mathfrak A}}
\nc{\frakb}{{\mathfrak b}} \nc{\frakB}{{\mathfrak B}}
\nc{\frakD}{{\mathfrak D}} \nc{\frakF}{\mathfrak{F}}
\nc{\frakf}{{\mathfrak f}} \nc{\frakg}{{\mathfrak g}}
\nc{\frakH}{{\mathfrak H}} \nc{\frakL}{{\mathfrak L}}
\nc{\frakM}{{\mathfrak M}} \nc{\bfrakM}{\overline{\frakM}}
\nc{\frakm}{{\mathfrak m}} \nc{\frakP}{{\mathfrak P}}
\nc{\frakN}{{\mathfrak N}} \nc{\frakp}{{\mathfrak p}}
\nc{\frakS}{{\mathfrak S}} \nc{\frakT}{\mathfrak{T}}
\nc{\frakX}{{\mathfrak X}} \nc{\frakx}{\mathfrak{x}}

\nc{\BS}{\mathbb{S
}}

\font\cyr=wncyr10 \font\cyrs=wncyr7
\nc{\li}[1]{\textcolor{red}{Li:#1}}
\nc{\tian}[1]{\textcolor{blue}{Tianjie: #1}}
\nc{\xing}[1]{\textcolor{purple}{Xing: #1}}
\nc{\revise}[1]{\textcolor{blue}{#1}}

\nc{\ID}{\mathfrak{I}} \nc{\lbar}[1]{\overline{#1}}
\nc{\bre}{{\rm bre}} \nc{\sd}{\cals} \nc{\rb}{\rm RB}
\nc{\A}{\rm angularly decorated\xspace} \nc{\LL}{\rm L}
\nc{\w}{\rm wid} \nc{\arro}[1]{#1} \nc{\tforall}{\text{ for all }}
\nc{\ver}{\rm ver}
\nc{\FN}{F_{\mathrm N}}
\nc{\FNA}{\FN(X)} \nc{\NA}{N_{X}}

\begin{document}

\title[Hopf algebras on free Nijenhuis algebras]{Left counital Hopf algebras on free Nijenhuis algebras}
%
\author{Xing Gao}
\address{School of Mathematics and Statistics, Key Laboratory of Applied Mathematics and Complex Systems, Lanzhou University, Lanzhou, Gansu 730000, P.\,R. China}
         \email{gaoxing@lzu.edu.cn}

\author{Peng Lei}
\address{
School of Mathematics and Statistics, Key Laboratory of Applied Mathematics and Complex Systems, Lanzhou University, Lanzhou, Gansu 730000, P.\,R. China}
\email{leip@lzu.edu.cn}

\author{Tianjie Zhang$^{*}$}\thanks{*Corresponding author}
\address{School of Mathematics and Statistics, Ningxia University, Yinchuan, Ningxia 750021, P.\,R. China}
         \email{tjzhangmath@nxu.edu.cn}

\date{\today}
\begin{abstract}
Factorization in algebra is an important problem.
In this paper, we first obtain a unique factorization in free Nijenhuis algebras.
By using of this unique factorization, we then define a coproduct and a left counital bialgebraic structure on a free Nijenhuis algebra.
Finally, we prove that this left counital bialgebra is connected and hence obtain a left counital Hopf algebra on a free Nijenhuis algebra.
\end{abstract}

\subjclass[2010]{
16W99, 
08B20 
16T10 
16T05  	
16T30  	
}

\keywords{Nijenhuis algebra, bracketed words, factorization, left counital bialgebra, left counital Hopf algebra}

\maketitle

\tableofcontents

\setcounter{section}{0}

\allowdisplaybreaks

\section{Introduction}
A {\bf Nijenhuis algebra} $(R,N)$ is an associative algebra $R$ equipped with a linear operator $N:R\to R$, called {\bf Nijenhuis operator},
satisfying the {\bf Nijenhuis equation}:
\begin{equation}
    N(u)N(v) = N(N(u)v) + N(u N(v)) - N^2(uv)\ \text{ for all } u, v\in R.
\mlabel{eq:Nij}
\end{equation}
The Lie algebra version of the associative Nijenhuis equation started in earnest in the 1950s~\mcite{N}.
In that paper, Nijenhuis introduced the crucial concept of Nijenhuis tensor, which fits closely into
the distinguished concepts of Schouten-Nijenhuis bracket, the Fr\"olicher-Nijenhuis bracket~\mcite{FN} and the Nijenhuis-Richardson bracket.
Thereafter, Magri et at.~\mcite{KM} studied the deformation of Lie brackets given by Nijenhuis operators.
Recently, Sokolov et al.~\mcite{GS1,GS2} examined the Nijenhuis operators on Lie algebras in the context of the classical Yang-Baxter equation,
which has a close relation with the Lie algebraic version of the Rota-Baxter equation recalled below~\mcite{BD,Gub}.

The Nijenhuis operator on an associative algebra can be originated to~\mcite{CGM} in the study of quantum bi-Hamiltonian systems.
In~\mcite{Far}, K. Ebrahimi-Fard gave the construction of a free commutative associative Nijenhuis algebra on a commutative associative algebra,
based on an augmented modified quasi-shuffle product. In~\mcite{U}, the associative Nijenhuis operators were constructed due to an application of a twisting operation on Hochschild complex by
analogy with Drinfeld's twisting operations. Latterly, Guo et al. constructed the free noncommutative associative Nijenhuis algebra on an algebra, and studied the associative Nijenhuis algebras with emphasis on the relationship between the category of Nijenhuis algebras and the category of NS algebras~\mcite{LG}.


Hopf algebras, named after Heinz Hopf, occur naturally in algebraic topology,
where they have broad connections with many areas in mathematics and mathematical physics~\mcite{Ab,CK,Sw}.
Much of the research in Hopf algebras are on specific classes of examples.
A crucial class of Hopf algebras is built from free objects in various contexts, such as
free associative algebras, the enveloping algebras of Lie algebras, and the free objects in the category of dendriform algebras of Loday and of tridendriform algebras of Loday and Ronco~\cite{LR}.
It is worth to note that the classical Connes-Kreimer Hopf algebra of rooted trees
can be considered as a free operated algebra~\cite{Guop,ZGG}.

The Nijenhuis equation can be viewed as the homogeneous version of the familiar {\bf Rota-Baxter equation}:
\begin{equation*}
P(u)P(v)=P(uP(v))+P(P(u)v)+\lambda P(uv),\ \text{ for all } u, v\in R,
\end{equation*}
which gives the Rota-Baxter operator $P$ on an associative algebra $R$.
Here $\lambda$, called the weight of $P$, is a prefixed element in the base ring of the algebra $R$.
See~\mcite{AGKO, Ba, G4, Gub} for further details and references.
Hopf algebra structures were equipped on free noncommutative Rota-Baxter algebras~\mcite{ZGG},
and Hopf algebra related structures were explored on free commutative Nijenhuis algebras~\mcite{ZG}.
Following closely these two inspiring works~\mcite{ZGG, ZG}, we equip Hopf type algebra related structures on free noncommutative Nijenhuis algebras in this paper,
by making use of the construction of free Nijenhuis algebras via bracketed words given in~\mcite{GGSZ,LG}.

Here is the layout of this paper. In Section~\ref{sec:RBHOPHAL}, we review
the construction of free noncommutative Nijenhuis algebra on a set in terms of bracketed words, and then
acquire a unique factorization of such bracketed words (Proposition~\mref{pp:cdiam}).
Using this unique factorization, we define a coproduct on a free Nijenhuis algebra, which, together with a left counity, turns the
free Nijenhuis algebra into a left counital cocycle bialgebra (Theorem~\ref{thm:main}), satisfying an one-cocycle condition (Eq.~\mref{eq:cpl}).
In the final Section~\mref{sec:hopf}, we grade the left counital bialgebra obtained in the previous section, and
prove that it is connected and hence a left counital Hopf algebra (Theorem~\mref{thm:Hfree}).
\smallskip

\noindent
{\bf Convention. } In this paper, all algebras are taken to be unitary associative (but not necessary commutative) over a unitary commutative ring $\bfk$, unless the contrary is specified. Also linear maps and tensor products are taken over $\bfk$. For any set $Y$, let $M(Y)$ and $S(Y)$ denote the free monoid and free semigroup generated by $Y$, respectively.

\section{Left counital bialgebra structures on free Nijenhuis algebras}
\label{sec:RBHOPHAL}
In this section, we give a left counital cocycle bialgebra structure on a free Nijenhuis algebra. We
first recall the construction of free Nijenhuis algebras by bracketed words.

\subsection{A free Nijenhuis algebras on a set}

\begin{defn}
Let $X$ be a set. A free Nijenhuis algebra on $X$ is a Nijenhuis algebra $(\FN(X),\NA)$ together with a map $j_X: X\to \FN(A)$ such that, for any  Nijenhuis algebra $(R, N)$ and any map $f:X\to R$, there is a unique
 Nijenhuis algebra homomorphism $\free{f}: \FN(X)\to R$
such that $\free{f}\circ j_X=f$:
$$ \xymatrix{ X \ar[rr]^{j_X}\ar[drr]_{f} && \FN(X) \ar[d]^{\free{f}} \\
&& R}
$$
\mlabel{de:fna}
\end{defn}

Free Nijenhuis algebras are quotients of free operated algebras, and the construction of free operated algebras
was given in~\cite{Gub, Guop}. We reproduce that construction here to review the notations.

Let $\lc Y\rc$ denote the set $\{ \lc y\rc \, |\, y\in Y\}.$
Thus $\lc Y\rc$ is a set indexed by $Y$ but disjoint with $Y$.
For a set $X$, we first let $\frakM_0:=M(X)$ be the free monoid generated by $X$, where the identity is
denoted by $\bfone$. Then we define
$\frakM_1:=M(X\cup \lc M(X)\rc)$
with $i_{0,1}$ being the natural injection
\begin{align*}
 i_{0,1}:& \frakM_0=M(X) \hookrightarrow
    \frakM_1=M(X\cup \lc \frakM_0\rc).
\end{align*}
We identify $\frakM_0$ with its image in $\frakM_1$. In particular, the identity $\bfone$ in $\frakM_0$ is sent to
$\bfone$ in $\frakM_1$.

Inductively assume that $\frakM_{n}$ has been defined
for $n\geq 1$, and define
\begin{equation*}
 \frakM_{n+1}:=M(X\cup \lc\frakM_{n}\rc ).
 \end{equation*}
Further assume that the embedding
$$i_{n-1,n}: \frakM_{n-1} \to \frakM_{n}$$
has been obtained. Then we have the injection
$$  \lc\frakM_{n-1}\rc \hookrightarrow
    \lc \frakM_{n} \rc.$$
Thus by the freeness of
$\frakM_{n}=M(X\cup \lc\frakM_{n-1}\rc)$, we have
\begin{eqnarray*}
\frakM_{n} &=& M(X\cup \lc\frakM_{n-1}\rc)\hookrightarrow
    M(X\cup \lc \frakM_{n}\rc) =\frakM_{n+1}.
\end{eqnarray*}
We finally define the monoid
$$ \frakM(X):=\dirlim \frakM_n = \bigcup_{n\geq 0}\frakM_n $$
with identity $\bfone$. Elements of $\frakM(X)$ are called {\bf bracketed words in $X$}.
Let $\bfk\frakM(X)$ be the free $\bfk$-module spanned by
$\frakM(X)$. Since the basis is a monoid, the
multiplication on $\frakM(X)$ can be extended via linearity to turn the
$\bfk$-module $\bfk\frakM(X)$ into an algebra, which
we still denote by $\bfk\frakM(X)$. Similarly, we can extend the operator
$\lc\ \rc: \frakM(X) \to \frakM(X)$, which takes $w \in \frakM(X)$ to
$\lc w\rc$, to an operator $\lc\ \rc$ on $\bfk\frakM(X)$ by linearity
and turn the algebra $\bfk\frakM(X)$ into an operated
algebra.

\begin{lemma}{\bf \cite{Gub, Guop}}
Let $i_X:X \to \frakM(X)$ and $j_X: \frakM(X) \to \bfk\frakM(X)$ be the natural embeddings. Then, with structures as above,
\begin{enumerate}
\item
the triple $(\frakM(X),\lc\ \rc, i_X)$ is the free operated monoid on $X$; and
\label{it:mapsetm}
\item
the triple $(\bfk\frakM(X),\lc\ \rc, j_X\circ i_X)$ is the free operated algebra
on $X$. \label{it:mapalgsg}
\end{enumerate}
\label{pp:freetm}
\end{lemma}

It was shown in~\cite{Gub} that every $w\in\frakM(X)\setminus\{\bfone\}$ has a unique {\bf standard decomposition:}
$$w = w_{1}\cdots w_{m},$$
where $w_{i},1\leq i\leq m$, is alternatively in the free semigroup $S(X)$ or in $\lc\frakM(X)\rc:=\{\lc u\rc\,|\,u\in \frakM(X) \}$.
We call $m$ the {\bf breadth} of $w$, denoted by $\bre(w)=m$. If $w=\bfone \in \frakM(X)$, we define $\bre(w) = 0$.
Elements $w\in \frakM_n\setminus \frakM_{n-1}$ are said to have {\bf depth} $n$, denoted by $\dep(w)=n$.

We now review a \bfk-basis of a free Nijenhuis algebra given in~\mcite{GGSZ,LG}, which is a subset of bracketed words.
Let $Y$ and $Z$ be subsets of $\frakM(X)$. Define the {\bf alternating products} of $Y$ and $Z$ by
\index{alternating product}
\allowdisplaybreaks{
\begin{eqnarray*}
\Lambda(Y,Z)&=&\Big( \bigcup_{r\geq 1} \big (Y\lc Z\rc \big)^r \Big) \bigcup
    \Big(\bigcup_{r\geq 0} \big (Y\lc Z\rc \big)^r  Y\Big) \notag \\
&& \bigcup \Big( \bigcup_{r\geq 1} \big( \lc Z\rc Y \big )^r \Big)
 \bigcup \Big( \bigcup_{r\geq 0} \big (\lc Z\rc Y\big )^r \lc Z\rc \Big) \bigcup \Big\{\bfone\Big\}.
\end{eqnarray*}
}
Recursively define $$\frak X_{0}:=M(X)\,\text{ and }\,\frak X_{n}:=\Lambda(\frak X_{0},\frak X_{n-1}),n\geq1.$$
Thus $\frak X_{0}\subseteq\cdots\subseteq\frak X_{n}\subseteq\cdots$ and so we have

$$\frak X_{\infty}:=\dirlim\frak X_{n} =\bigcup_{n\geq 0} \frak X_{n}.$$

Let $\FNA:=\bfk\frak X_{\infty}$ the free \bfk-module spanned by $\frak X_{\infty}$. To make $\FNA$ into a Nijenhuis algebra, a Nijenhuis operator $\NA$ and a product $\diamond$
need to be equipped.
Let $w,w'$ be two basis elements in $\frak X_{\infty}$. Define a linear operator $\NA: \FNA\to \FNA$ by setting
$$\NA(w)=\lc w\rc.$$
Next we define $w~\diamond~w'$ inductively on the sum $n:=\dep(w)+\dep(w')\geq 0$.
If $n=0$, then $w,w'\in\frak X_{0}=M(X)$ and define $w~\diamond~w':=xx'$, the concatenation in $M(X)$.
Suppose that $w~\diamond~w'$ have been defined for all $w,w'\in\frak X_{\infty}$
with $n\leq k$ for a $k\geq0$ and consider $w,w'\in\frak X_{\infty}$ with $n=k+1$.
If $w = \bfone$ or $w'=\bfone$, without loss of generality, letting
$w = \bfone$, then define $w~\diamond~w' := w'$. Assume $w \neq \bfone$ and $w' \neq \bfone$, and so
$\bre(w), \bre(w')\geq 1$. Consider first $\bre(w)=\bre(w')=1$. Then $w$ and $w'$ are in $S(X)\subset\frak X_{0}$ or
$\lc\frak X_{\infty}\rc$ and can not be both in $S(X)$ since $n=k+1\geq1$. We accordingly define
\begin{equation}
\begin{aligned}
w~\diamond~w'=\left\{\begin{array}{ll}
ww',&\text{ if }w\in S(X),\,w'\in\lc\frak X_{\infty}\rc,\\
ww',&\text{ if }w\in\lc\frak X_{\infty}\rc,\,w'\in S(X),\\
\lc w~\diamond~\lbar{w}'\rc+\lc \lbar{w}~\diamond~ w'\rc-\lc\lc\lbar{w}~\diamond~\lbar{w}'
\rc\rc,&\text{ if }w=\lc\lbar{w}\rc,\,w'=\lc\lbar{w}'\rc\in\lc\frak X_{\infty}\rc.
\end{array}\right.
\end{aligned}
\mlabel{eq:Bdia}
\end{equation}
Here the product in the first two cases is by concatenation and in the third case is by the induction hypothesis on $n$.
Now consider $\bre(w)\geq1$ or $\bre(w')\geq1$. Let $w=w_{1}\cdots w_{m}$ and $w'=w'_{1}\cdots w'_{m'}$ be the standard decompositions of $w$ and $w'$. Define
\begin{equation}
w~\diamond~ w':=w_{1}\cdots w_{m-1} (w_{m}~\diamond~ w'_{1})w'_{2}\cdots w'_{m'},
\mlabel{eq:cdiam}
\end{equation}
where $w_{m}~\diamond~ w'_{1}$ is defined by Eq.~(\mref{eq:Bdia}) and the rest products are
given by the concatenation.


\begin{lemma}\mcite{GGSZ,LG}
Let $X$ be a set and $j_X:X\to \frak X_\infty \to \FN(X)$ the natural injection. Then the quadruple $(\FN(X),\shpr,N_X,j_X)$ is the free Nijenhuis algebra on $X$.
\mlabel{lem:ncfree}
\end{lemma}

{\em For the rest of this paper}, unless alternative notations are specifically given, we will use the
infix notation $\lc w\rc$ interchangeably with $\NA(w)$ for any $w\in \FNA$.
Write
$$\ID:=X \sqcup \lc \frak X_\infty\rc.
$$
\begin{defn}
We call a sequence $w_1,\cdots,w_m$ from the set $\ID$ {\bf alternating} if
no consecutive elements in the sequence are in $\lc \frak X_\infty\rc$, that is,
either $w_i$ or $w_{i+1}$ is in $X$ for each $1\leq i\leq m-1$.
\mlabel{de:alt}
\end{defn}

The following unique factorization sheds insight on the construction of a coproduct on $\FNA$.
\begin{prop}
Let $w\in \frak X_\infty \setminus\{\bfone\}$. Then there
is a unique alternating sequence $w_1,\cdots,w_m$ in $\ID$ such that
\begin{equation*}
w =w_{1}~\diamond~\cdots~\diamond~w_{m},
\end{equation*}
called the {\bf \diam factorization} of $w$. We also call $\w(w):=m$ the {\bf width} of $w$.
If $w=\bfone$, we define $\w(w) := 0$.
\mlabel{pp:cdiam}
\end{prop}

\begin{proof}
(Existence). Suppose $w = u_1\cdots u_n$ is the standard decomposition of $w$.
Then $u_1,\cdots, u_n$ are alternatively in $S(X)$ and $\lc \frakX_\infty\rc$. Expanding the factors that are in $S(X)$,
we get
$$w=w_1\cdots w_m, \,\text{ where $w_1,\cdots,w_m$ is an alternating sequence}.$$
It follows from Eq.~(\mref{eq:Bdia}) that $w = w_1~\diamond~\cdots ~\diamond~ w_m$.

(Uniqueness). Suppose $w_{1},\cdots,w_{m}$ and
$w'_{1},\cdots,w'_{m'}$ are two alternating sequences such that
$$w =w_{1}~\diamond~\cdots~\diamond~w_{m}
=w'_{1}~\diamond~\cdots~\diamond~w'_{m'}.$$
By Definition \mref{de:alt} and Eq.~(\mref{eq:Bdia}), we get
$$w_{1}\cdots w_{m} = w_{1} ~\diamond~\cdots~\diamond~ w_{m}
=w =  w'_{1}~\diamond~\cdots~\diamond~ w'_{m'}
= w'_{1} \cdots w'_{m'},$$
and so
$m=m'$ and $w_{i} = w'_{i}$ for $1\leq i\leq m$ by the construction of $\frak X_\infty$.
\end{proof}

\begin{remark}
\begin{enumerate}
\item
The breadth and width of $w\in \frakX_\infty$ are different. For example, let $w=x_1x_2x_3$ with $x_1, x_2,x_3\in X$.
Then $\bre(w)=1$ and $\w(w)=3$.

\item If the sequence $w_1,\cdots,w_m$ is alternating, then, for each $1\leq i\leq m-1$,  $w_i$ and $w_{i+1}$ cannot be both in $\lc \frakX_\infty\rc$. However, $w_i$ and $w_{i+1}$ may be both in $X$.
\end{enumerate}
\mlabel{remark:1}
\end{remark}

\subsection{The left counital bialgebra structure on free Nijenhuis algebras}
\mlabel{sec:bial}
In this subsection, we use Proposition~\mref{pp:cdiam} to get a coproduct on $\FNA$,
which, together with a left counit, makes $\FNA$ into a left counital coalgebra.
The 1-cocycle condition $\Delta P=P\ot \bfone +(\id \ot P)\Delta$, which is used to construct the Hopf
algebra structure on free Rota-Baxter algebras~\cite{ZGG}, does not work for free Nijenhuis algebras.
So Guo et al. proposed the following concepts~\mcite{ZG}.

\begin{defn}
\begin{enumerate}
\item
A {\bf left counital coalgebra} is a triple $(H,\Delta,\vep)$, where the coproduct $\Delta:H\to H\ot H$ satisfies the coassociativity: $(\Delta\ot\id)\Delta=(\id\ot\Delta)\Delta$ and the left counit $\vep:H\to \bfk$ satisfies the {\bf left  counicity}: $(\vep\ot\id)\Delta =\beta_\ell$, where $\beta_\ell:H\to \bfk\ot H, u\mapsto 1\ot u\ \tforall u\in H$.

\item
A {\bf left counital bialgebra} is a quintuple $(H,m, u, \Delta,\vep)$, where $(H, m,u)$ is an algebra and $(H, \Delta, \vep)$
is a left counital coalgebra such that $\Delta$ and $\vep$ are homomorphisms of algebras.

\item
A {\bf left counital operated bialgebra} is a sextuple $(H,m,u,\Delta,\vep,P)$, where $(H,m, u, \Delta,\vep)$ is a left counital bialgebra and $(H,m,u,P)$ is an operated algebra, that is, an algebra $(H, m, u)$ with a linear operator $P:H\to H$.
\item
A {\bf left  counital cocycle bialgebra} is a left counital operated bialgebra $(H,m,u,\Delta,\vep,P)$ satisfying the one-cocycle property:
\begin{equation}
\Delta P=(\id\ot P)\Delta.
\mlabel{eq:cpl}
\end{equation}
\end{enumerate}
\end{defn}

Thanks to Eq.~(\mref{eq:cpl}), we can construct a coproduct $\Delta:\FNA\to \FNA\ot \FNA$
by defining $\Delta(w)$ for $w\in\frak X_{\infty}$ through an induction on the depth $\dep(w)$.
If $\dep(w)=0$, then $w\in\frak X_{0}=M(X)$. First we define
\begin{equation}
\Delta(w):=\bfone\ot \bfone\,\text{ provided }\, w= \bfone
\mlabel{eq:Init}
\end{equation}
and
\begin{equation}
\Delta(w):= \bfone\ot x  \,\text{ provided }\, w=x\in X.
\mlabel{eq:Dbull}
\end{equation}
If $ w=x_{1}\cdots x_{m}\in S(X)$ with $m\geq2$ and $x_{i}\in X$ for $1\leq i\leq m$,
then $w = x_1 ~\diamond~ \cdots ~\diamond~ x_m $ by Eq.~(\mref{eq:Bdia}), and we define
\begin{equation}
\Delta(w):=\Delta(x_1)~\diamond~ \cdots ~\diamond~ \Delta(x_{m}).
\mlabel{eq:xexte}
\end{equation}

Assume $\Delta(w)$ has been defined for $w\in\frak X_{\infty}$ with $\dep(w)\leq n$ and consider $w\in \frak X_{\infty}$ with $\dep(w)=n+1$.
In view of Proposition~\mref{pp:cdiam}, there is a unique alternating sequence $w_1,\cdots, w_m$ such that
$$w = w_{1} ~\diamond~ \cdots  ~\diamond~ w_{m}, \text{ where } w_1, \cdots, w_m \in \ID, m\geq 1.$$
If $m=1$, then $w\in \lc\frak X_{\infty}\rc$ by $\dep(w)=n+1\geq1$. So we can write $w:=\lc\lbar{w}\rc$ for some $\lbar{w}\in \frak X_{\infty}$.
Then we define by Eq.~(\mref{eq:cpl}) that
\begin{equation}
\Delta(w):=\Delta(\lc\lbar{w}\rc):=(\id\ot \NA)\Delta(\lbar{w}),
\mlabel{eq:Tree}
\end{equation}
where $\Delta(\lbar{w})$ is defined by the induction hypothesis.
If $m\geq 2$, we define
\begin{equation}
\Delta(w):=\Delta(w_{1})~\diamond~\cdots~\diamond~ \Delta(w_{m}),
\mlabel{eq:Forest}
\end{equation}
where $\Delta(w_{1}), \cdots, \Delta(w_{m})$ are defined in Eq.~(\mref{eq:Dbull}) or Eq.~(\mref{eq:Tree}).
Note $\Delta(w)$ is well-defined by the uniqueness of the \cdiam factorization of $w$ in Proposition~\mref{pp:cdiam}.
This completes the inductive definition of the coproduct $\Delta$ on $\FNA$.
We shall tacitly denote by $\Delta$ the coproduct defined here throughout the remainder of the paper.

The following is an easy property of $\Delta$, which will be used frequently.

\begin{lemma}
Let $X$ be a set and $w\in M(X)$. Then $\Delta(w) =\bfone\ot w.$
\mlabel{lem:delw}
\end{lemma}

\begin{proof}
Consider first $w=\bfone$. Then $\Delta(w) = \bfone\ot \bfone = \bfone\ot w$ by Eq.~(\mref{eq:Init}).
Consider next $w\neq \bfone$. Then we may write $w=x_{1}\cdots x_{m}$, where $m\geq1$ and $x_{1},\cdots, x_m \in X$.
Then it follows from Eqs.~(\mref{eq:Dbull}) and~(\mref{eq:xexte}) that
\begin{align*}
\Delta(w) =& \Delta(x_1\cdots x_m) = \Delta(x_1) ~\diamond~ \cdots ~\diamond~ \Delta(x_m) \\
=& (\bfone\ot x_1)  ~\diamond~ \cdots ~\diamond~ (\bfone\ot x_m) =\bfone \ot (x_1 ~\diamond~ \cdots ~\diamond~ x_m) \\
=&  \bfone \ot (x_1 \cdots  x_m) = \bfone\ot w,
\end{align*}
as required.
\end{proof}

We are going to show the compatibility of $\Delta$ with the product $\diamond$.
Let us begin with a simple case.

\begin{lemma}
Let $X$ be a set and $u,v\in M(X)$.
Then
\begin{equation}
\Delta(u~\diamond~ v)=\Delta(u)~\diamond~\Delta(v).
\notag
\end{equation}
\mlabel{lem:I1}
\end{lemma}

\begin{proof}
We have two cases to consider.

\noindent{\bf Case 1.} $u = \bfone$ or $v = \bfone$. By symmetry, let $u = \bfone$.
Then $\Delta(u) =\bfone\ot \bfone$ by Eq.~(\mref{eq:Init}) and so
$$\Delta(u~\diamond~v)= \Delta(v) = (\bfone\ot \bfone) ~\diamond~\Delta(v) =\Delta(u)~\diamond~\Delta(v).$$

\noindent{\bf Case 2.} $u \neq \bfone$ and $v \neq \bfone$.
Then we can write
$$u= x_{1} \cdots x_{p}\,\text{ and }\, v=y_{1} \cdots y_{q},$$
where $p,q\geq 1$ and $x_1,\cdots, x_p, y_1, \cdots ,y_q\in X$.
So
\begin{align*}
\Delta(u~\diamond~ v) =& \Delta(x_1\cdots x_p y_1\cdots y_q) \quad \text{(by Eq.~(\mref{eq:Bdia}))}\\
=& \bfone \ot (x_1\cdots x_p y_1\cdots y_q) \quad \text{(by Lemma~\mref{lem:delw})}\\
=& \bfone \ot ((x_1\cdots x_p)~\diamond~  (y_1\cdots y_q)) \quad \text{(by Eq.~(\mref{eq:cdiam}))}\\
=&  (\bfone \ot (x_1\cdots x_p) )~\diamond~(\bfone \ot (y_1\cdots y_q))\\
=& \Delta (x_1\cdots x_p)~\diamond~\Delta (y_1\cdots y_q)  \quad \text{(by Lemma~\mref{lem:delw})} \\
=& \Delta (u)~\diamond~\Delta (v),
\end{align*}
as required.
\end{proof}

\begin{lemma}
Let $X$ be a set and $u, v\in\FNA$. Then
\begin{equation}
\Delta(u~\diamond~ v)=\Delta(u)~\diamond~\Delta(v).
\mlabel{eq:Morphism}
\end{equation}
\mlabel{lem:calgh}
\end{lemma}

\begin{proof}
Since $\Delta$ is linear and $\diamond$ is bilinear, we only need to prove the result for
basis elements  $u, v\in \frak X_{\infty}$.
If $u=\bfone$ or $v = \bfone$, without loss of generality, letting $u=\bfone$, then by Eq.~(\mref{eq:Init}),
$$\Delta(u~\diamond~ v)= \Delta(v) = (\bfone \ot \bfone) ~\diamond~ \Delta(v) = \Delta(u)~\diamond~\Delta(v).$$
Assume $u\neq \bfone$ and $v\neq \bfone$, and we proceed to prove this case by using induction on the sum of the depths
$$n:=\dep(u) + \dep(v).$$
For the initial step of $n=0$, we have $\dep(u)=\dep(v) =0$ and so $u,v\in M(X)$.
Then the result is valid from Lemma~\mref{lem:I1}.

For the inductive step, assume the result holds for $n\leq k$ for a $k\geq 0$, and consider the case of $n=k+1$.
We reduce to prove the result by induction on the sum of the widths
$$m:=\w(u)+\w(v)\geq 2.$$

If $m=2$, then $\w(u)=\w(v)=1$. Consider first $u\in X$ or $v\in X$. Then the result holds by Eq.~(\mref{eq:Forest}).
Consider next $u\notin X$ and $v\notin X$. Then by $\w(u)=\w(v)=1$, we can write
$$u:=\lc\lbar{u}\rc\,\text{ and }\, v :=\lc\lbar{v}\rc \,\text{ for some } \lbar{u},\lbar{v}\in \frak X_{\infty}.$$
Using the Sweedler notation, we denote by
$$\Delta(\lbar{u})=\sum_{(\lbar{u})}\lbar{u}_{(1)}\otimes \lbar{u}_{(2)} \text{ and }
\Delta(\lbar{v})=\sum_{(\lbar{v})}\lbar{v}_{(1)}\otimes \lbar{v}_{(2)}.$$
Thus we have
\allowdisplaybreaks{
\begin{align*}
&\Delta(u~\diamond~ v) = \Delta\left(\lc\lbar{u}\rc~\diamond~ \lc\lbar{v}\rc\right)
= \Delta(\lc\lbar{u}~\diamond~ \lc\lbar{v}\rc+\lc\lbar{u}\rc~\diamond~ \lbar{v}
-\lc\lbar{u}~\diamond~\lbar{v}\rc\rc) \\
=&\,(\id\ot \NA)\Delta\left(\lbar{u}~\diamond~ \lc\lbar{v}\rc+\lc\lbar{u}\rc~\diamond~ \lbar{v}
-\lc\lbar{u}~\diamond~\lbar{v}\rc\right)\quad (\text{by Eq.~(\mref{eq:Tree})})\\
=&\,(\id\ot  \NA)
\left(\Delta(\lbar{u}) ~\diamond~\Delta(\lc\lbar{v}\rc)+\Delta(\lc\lbar{u}\rc)~\diamond~\Delta(\lbar{v})
-\Delta(\lc\lbar{u}~\diamond~\lbar{v}\rc)\right)\quad(\text{by the induction hypothesis on~}n)\\
=&\,(\id\ot  \NA)\bigg(\Delta(\lbar{u})~\diamond~((\id\ot  \NA)\Delta(\lbar{v}))
+((\id\ot \NA)\Delta(\lbar{u}))~\diamond~\Delta(\lbar{v})-(\id\ot \NA)(\Delta(\lbar{u})
~\diamond~\Delta(\lbar{v}))\bigg)
\quad(\text{by Eq.~(\mref{eq:Tree})})\\
=&\,(\id\ot  \NA)\bigg(\sum_{(\lbar{u}),\,(\lbar{v}) }
(\lbar{u}_{(1)}~\diamond~ \lbar{v}_{(1)})\ot (\lbar{u}_{(2)}~\diamond~ \lc\lbar{v}_{(2)}\rc)
+\sum_{(\lbar{u}),\,(\lbar{v}) } (\lbar{u}_{(1)}~\diamond~ \lbar{v}_{(1)})\ot
(\lc\lbar{u}_{(2)}\rc~\diamond~ \lbar{v}_{(2)}) \\
& - \sum_{(\lbar{u}),\,(\lbar{v}) }
(\lbar{u}_{(1)}~\diamond~ \lbar{v}_{(1)})\ot\lc\lbar{u}_{(2)}~\diamond~ \lbar{v}_{(2)}\rc\bigg)\\
=&\,\sum_{(\lbar{u}),\,(\lbar{v}) }
(\lbar{u}_{(1)}~\diamond~ \lbar{v}_{(1)})\ot \bigg(\lc \lbar{u}_{(2)}~\diamond~ \lc\lbar{v}_{(2)}\rc \rc +  \lc \lc\lbar{u}_{(2)}\rc~\diamond~ \lbar{v}_{(2)} \rc - \lc \lc\lbar{u}_{(2)}~\diamond~ \lbar{v}_{(2)}\rc\rc \bigg)\\
=&\,\sum_{(\lbar{u}),\,(\lbar{v}) }
(\lbar{u}_{(1)}~\diamond~ \lbar{v}_{(1)})\ot (\lc\lbar{u}_{(2)}\rc~\diamond~ \lc\lbar{v}_{(2)}\rc)\quad(\text{by Eq.~}(\mref{eq:Bdia}))\\
=&\,\left(\sum_{(\lbar{u})}\lbar{u}_{(1)}\ot \lc\lbar{u}_{(2)}\rc\right)~\diamond~
\left(\sum_{(\lbar{v})} \lbar{v}_{(1)}\ot \lc\lbar{v}_{(2)}\rc\right)\\
=&\,((\id\ot \NA)\Delta(\lbar{u}))~\diamond~((\id\ot \NA)\Delta(\lbar{v}))\\
=&\,\Delta(\lc\lbar{u}\rc)~\diamond~\Delta(\lc\lbar{v}\rc)
\quad(\text{by Eq.~(\mref{eq:Tree})})\\
=&\, \Delta(u)~\diamond~\Delta(v).
\end{align*}
}
This completes the initial step of the induction on $\w(u)+\w(v)$.

Assume that Eq.~(\mref{eq:Morphism})
holds for the case of
$n= k+1$ and $2\leq m \leq \ell$,
and consider the case of $n=k+1$ and $m=\ell+1$. Then $m\geq3$ and so either $\w(u)\geq 2$ or $\w(v)\geq 2$.
We have three cases to prove:
\begin{enumerate}
\item  $\w(u)\geq2,\w(v)\geq2$;

\item $\w(u)\geq2,\w(v)=1$;

\item $\w(u)=1,\w(v)\geq2$.
\end{enumerate}
We only prove the first case, because other cases are easier and can be treated in the same way.
So suppose
$$\w(u)\geq2\,\text{ and }\, \w(v)\geq2.$$
From Proposition~\mref{pp:cdiam}, there are unique alternating sequences $u_{1}, \cdots,u_{p}$ and $v_{1},\cdots ,v_{q}$ such that
$$u=u_{1}~\diamond~\cdots ~\diamond~ u_{p}\,\text{ and }\,v=v_{1}~\diamond~\cdots ~\diamond~ v_{q},$$
and so
$$u~\diamond~ v = u_{1}~\diamond~\cdots ~\diamond~u_{p-1}~\diamond~(u_{p}\diamond v_{1}) ~\diamond~v_{2}\cdots ~\diamond~v_{q}.$$
If $u_{p}\notin \lc \frak X_\infty\rc$ or $v_1\notin \lc \frak X_\infty\rc$,
then by Eq.~(\mref{eq:Forest}),
\begin{align*}
\Delta(u~\diamond~ v) = & \Delta(u_{1}~\diamond~ \cdots ~\diamond~u_{p} ~\diamond~
v_{1}~\diamond~ \cdots ~\diamond~v_{q})\\
=& \Delta(u_{1}) ~\diamond~ \cdots ~\diamond~ \Delta(u_{p}) ~\diamond~
\Delta(v_{1})~\diamond~ \cdots ~\diamond~ \Delta(v_{q})\\
=& \Delta(u_{1} ~\diamond~ \cdots ~\diamond~ u_{p}) ~\diamond~
\Delta(v_{1}~\diamond~ \cdots \, ~\diamond~\,v_{q})\\
=& \Delta(u) ~\diamond~ \Delta(v).
\end{align*}
If $u_{p} \in \lc \frak X_\infty\rc$ and $v_1 \in \lc \frak X_\infty\rc$, let
$$u_{p}=\lc\lbar{u}_{p}\rc\,\text{ and }\,v_{1}=\lc\lbar{v}_{1}\rc \, \text{ for some } \lbar{u}_{p},\lbar{v}_{1}\in \frak X_\infty.$$
By Eq.~(\mref{eq:Bdia}), we may suppose
$$u_{p}~\diamond~ v_{1}= \sum_{i} c_i \lc w_{i}\rc ,\text{ where each }\, c_i\in \bfk\setminus\{0\}, w_{i}\in\frak X_{\infty}.$$
So we get
\begin{align*}
\Delta(u~\diamond~ v)
&=\Delta(u_{1}~\diamond~\cdots ~\diamond~u_{p-1}~\diamond~(u_{p}\,~\diamond~ \, v_{1})~\diamond~v_{2}~\diamond~\cdots ~\diamond~v_{q})\\
%
%
&= \sum_{i}c_i \Delta(u_{1} ~\diamond~ \cdots ~\diamond~ u_{p-1} ~\diamond~ \lc w_{i}\rc
 ~\diamond~ v_{2} ~\diamond~ \cdots ~\diamond~ v_{q})\\
&=\sum_{i}c_i \Delta(u_{1})~\diamond~\cdots~\diamond~ \Delta(u_{p-1})~\diamond~\Delta(\lc w_{i}\rc)~\diamond~\Delta( v_{2})~\diamond~ \cdots~\diamond~ \Delta(v_{q})  \quad\text{(by Eq.~(\mref{eq:Forest}))}\\
&=\Delta(u_{1})~\diamond~\cdots~\diamond~ \Delta(u_{p-1})~\diamond~\Delta(u_{p}~\diamond~ v_{1})~\diamond~\Delta( v_{2})~\diamond~ \cdots~\diamond~ \Delta(v_{q})\\
&=\Delta(u_{1})~\diamond~\cdots~\diamond~ \Delta(u_{p-1})~\diamond~\Delta(u_{p})~\diamond~ \Delta(v_{1})~\diamond~\Delta( v_{2})~\diamond~ \cdots~\diamond~ \Delta(v_{q})\quad(\text{by the case of } m=2)\\
&=\Delta(u_{1}~\diamond~\cdots~\diamond~ u_{p})~\diamond~ \Delta(v_{1}~\diamond~\cdots~\diamond~ v_{q}) \quad\text{(by Eq.~(\mref{eq:Forest}))}\\
&=\Delta(u)~\diamond~\Delta(v).
\end{align*}
This completes the induction on $\w(u)+\w(v)=m$ and so the induction on $\dep(u)+\dep(v)=n$.
\end{proof}

\begin{lemma}
Let $X$ be a set. Then $\Delta$ is a coproduct on $\FNA$, that is, $\Delta$ satisfies the coassociative law
\begin{equation}
(\Delta\ot\id)\Delta(w)=(\id\ot\Delta)\Delta(w) \,\text{ for }\, w\in\FNA.
\mlabel{eq:coass}
\end{equation}
\mlabel{lem:coasso}
\end{lemma}

\begin{proof}
We prove the result by induction on depth $n:=\dep(w)$ of the basis elements $w\in \frakX_\infty$.
For the initial step of $n=0$, we have $w\in M(X)$.
By Lemma~\mref{lem:delw}, we get $\Delta(w)= \bfone\ot w$ and so
\begin{align*}
(\id\otimes\Delta)\Delta(w)=&\ (\id\otimes\Delta) (\bfone \ot w) = \bfone \ot \Delta(w) =\bfone\ot \bfone \ot w \\
=& \ \Delta(\bfone) \ot w =(\Delta\ot\id)(\bfone \ot w) = (\Delta\ot\id)\Delta(w).
\end{align*}

For the inductive step, assume the result is true for the case of $n\leq k$ for a $k\geq 0$,
and consider the case of $n=k+1$. We next reduce to prove the result by induction on the width $m:=\w(w)$.
Since $\dep(w) = k+1 \geq 1$, we have $w\neq \bfone$ and so
$m=\w(w)\geq 1$. If $m=1$, since $\dep(w)\geq 1$, we can write
$$w=\lc \overline{w}\rc = \NA(\overline{w}) \,\text{ for some } \overline{w}\in \frakX_k.$$
So we get
\begin{align*}
(\id\ot \Delta)\Delta(w)
&=(\id\ot \Delta)\Delta(\NA(\lbar{w}))\\
&=(\id\ot\Delta)(\id\ot \NA)\Delta(\lbar{w})\quad (\text{by Eq.~(\mref{eq:Tree})})\\
&=(\id\ot (\Delta \NA))\Delta(\lbar{w})\quad (\text{by Eq.~(\mref{eq:Init})})\\
&= \Big(\id\ot \big( (\id \ot \NA)\Delta \big) \Big)\Delta(\lbar{w})\quad (\text{by Eq.~(\mref{eq:Tree})})\\
&=(\id\ot\id\ot \NA)(\id\ot\Delta)\Delta(\lbar{w})\\
&=(\id\ot\id\ot \NA)(\Delta\ot\id)\Delta(\lbar{w})\quad(\text{by the induction hypothesis on~}n)\\
&=(\Delta\ot \NA)\Delta(\lbar{w})\\
&=(\Delta\ot \id)(\id\ot \NA)\Delta(\lbar{w})\\
&=(\Delta\ot\id)\Delta(w)\quad (\text{by Eq.~(\mref{eq:Tree})}).
\end{align*}
Assume that the result is valid for the case of
$n=k+1$ and $m\leq \ell$, and consider the case of $n=k+1$ and $m=\ell+1\geq 2$.
So we can suppose
$$w=w'_1 w'_{2}=w'_1~\diamond~w'_{2},$$
where
$$\w(w'_{1})+\w(w'_{2})=\ell+1 \,\text{ and }\, 1\leq \w(w'_{1}),\,\w(w'_{2})\leq\ell.$$
Using the Sweedler notation, we can write
$$\Delta(w'_{1})=\sum^{}_{(w'_{1})}w'_{1(1)}\ot w'_{1(2)}\,\text{ and }\,
\Delta(w'_{2})=\sum^{}_{(w'_{2})}w'_{2(1)}\ot w'_{2(2)}.$$
By the induction hypothesis on $m$, we have
$$(\Delta\ot\id)\Delta(w'_{1})
=(\id\ot\Delta)\Delta(w'_{1})\,\text{ and }\,(\Delta\ot\id)\Delta(w'_{2})
=(\id\ot\Delta)\Delta(w'_{2}),$$
that is,
\begin{equation}
\sum^{}_{(w'_{1})}\Delta(w'_{1(1)})\ot w'_{1(2)}
=\sum^{}_{(w'_{1})}w'_{1(1)}\ot\Delta(w'_{1(2)})\,\text{ and }\,
\sum^{}_{(w'_{2})}\Delta(w'_{2(1)})\ot w'_{2(2)}
=\sum^{}_{(w'_{2})}w'_{2(1)}\ot\Delta( w'_{2(2)}).
\mlabel{eq:coass0}
\end{equation}
So we get
\begin{align*}
(\id\ot\Delta)\Delta(w)
&=(\id\ot\Delta)\Delta(w'_{1}~\diamond~ w'_{2})\\
&=(\id\ot\Delta)(\Delta(w'_{1})~\diamond~\Delta(w'_{2}))
\quad(\text{by Eq.~(\mref{eq:Morphism})})\\
&=\sum^{}_{(w'_{1})}\sum^{}_{(w'_{2})}(w'_{1(1)}~\diamond~ w'_{2(1)})\ot\Delta(w'_{1(2)}~\diamond~ w'_{2(2)})
\quad(\text{by linearity})\\
&=\sum^{}_{(w'_{1})}\sum^{}_{(w'_{2})}(w'_{1(1)}~\diamond~ w'_{2(1)})\ot(\Delta(w'_{1(2)})~\diamond~\Delta(w'_{2(2)}))
\quad(\text{by Eq.~(\mref{eq:Morphism})})\\
&=\left(\sum^{}_{(w'_{1})}w'_{1(1)}\ot\Delta(w'_{1(2)})\right)~\diamond~
\left(\sum^{}_{(w'_{2})}w'_{2(1)}\ot\Delta( w'_{2(2)})\right).
\end{align*}
With the similar argument, we can get
$$ (\Delta\ot\id)\Delta(w) =
\left(\sum^{}_{(w'_{1})}\Delta(w'_{1(1)})\ot w'_{1(2)}\right)~\diamond~
\left(\sum^{}_{(w'_{2})}\Delta( w'_{2(1)})\ot w'_{2(2)}\right).$$
Thus Eq.~(\mref{eq:coass}) holds by Eq.~(\mref{eq:coass0}).
This completes the inductive proof.
\end{proof}

We turn to construct a left counit on $\FNA$. Define a linear map:
\begin{equation}
\vep: \FNA\rightarrow\bfk, \,  \bfone \mapsto 1_\bfk, \, w \mapsto 0  \,\text{ for } w\in \frak X_{\infty}\setminus\{\bfone\}.
\mlabel{eq:vep}
\end{equation}
The next result shows that $\varepsilon$ is compatible with the product $\diamond$.

\begin{lemma}
Let $X$ be a set and $u, v\in\FNA$. Then
\begin{equation*}
\varepsilon(u~\diamond~ v)=\varepsilon(u)\varepsilon(v).
\end{equation*}
\mlabel{lem:counitprod}
\end{lemma}

\begin{proof}
By linearity of $\varepsilon$, it is sufficient to prove the result for basis elements $u, v\in \frak X_\infty$.
If $u= \bfone$ or $v=\bfone$, by symmetry, let $u= \bfone$. Then $\vep(u) = 1_\bfk$ and
$$\varepsilon(u~\diamond~ v) =\varepsilon(v) = 1_\bfk \varepsilon(v)=\varepsilon(u)\varepsilon(v).$$
Suppose $u\neq \bfone$ and $v\neq \bfone$. Then $u~\diamond~v\neq \bfone$ and so by Eq.~(\mref{eq:vep})
$$ \varepsilon(u~\diamond~ v) =0 = \varepsilon(u)\varepsilon(v),$$
as required.
\end{proof}

\begin{lemma}
Let $X$ be a set. Then the $\vep$ given in Eq.~(\mref{eq:vep}) is a left counit on $\FNA$, that is, $\vep$ satisfies the left counicity
\begin{equation*}
(\vep\ot\id)\Delta(w)=\beta_{\ell}(w)\,\text{ for } w\in\FNA,
\end{equation*}
where $\beta_\ell:\FNA\rightarrow \bfk\ot\FNA$ is given by $w\mapsto 1_{\bfk}\ot w$.
\mlabel{lem:counitt}
\end{lemma}

\begin{proof}
We prove the result by induction on depth $n:=\dep(w)$ of basis elements $w\in \frakX_\infty$.
For the initial step of $n=0$, we have $w\in M(X)$.
By Lemma~\mref{lem:delw}, we obtain $\Delta(w) = \bfone \ot w$
and so
\begin{align*}
(\vep\ot\id)\Delta(w)=(\vep\ot\id) (\bfone \ot w) =\vep(\bfone)\ot  w
=1_{\bfk}\ot w =\beta_\ell(F).
\end{align*}

For the induction step, assume the result is true for the case of $n\leq k$ for a $k\geq 0$, and consider the case of $n=k+1$.
We next reduce to prove the result by induction on the width $m:=\w(w)$.
Since $\dep(w) = k+1 \geq 1$, we have $w\neq \bfone$ and so
$m=\w(w)\geq 1$. If $m=1$, since $\dep(w)\geq 1$, we can write
$$w=\NA(\overline{w})\,\text{ for some } \overline{w}\in \frakX_k,$$
which implies
\begin{align*}
(\vep\ot \id )\Delta(w)
&= (\vep\ot\id )\Delta(\NA(\lbar{w}))\\
&= (\vep\ot \id )(\id\ot \NA)\Delta(\lbar{w})\quad(\text{by Eq.~(\mref{eq:Tree})})\\
&=(\id\ot \NA)(\vep\ot \id )\Delta(\lbar{w})\\
&= (\id\ot \NA)(1_{\bfk}\ot\lbar{w})
\quad( \text{by the induction hypothesis on~}n)\\
&= 1_{\bfk}\ot w =\beta_\ell(w).
\end{align*}

Assume the result holds for the case of $n=k+1$ and $m\leq \ell$, and consider
the case of $n=k+1$ and $m=\ell+1\geq 2$. So we can write
$$w=w_{1}~\diamond~ w_{2},$$
where
$$\w(w_{1})+\w(w_{2})=\ell+1\,\text{ and }\, 1\leq \w(w_{1}),\,\w(w_{2})\leq\ell.$$
Using the Sweedler notation, we may write
$$\Delta(w_{1})=\sum^{}_{(w_{1})}w_{1(1)}\ot w_{1(2)}\,\text{ and }\,
\Delta(w_{2})=\sum^{}_{(w_{2})}w_{2(1)}\ot w_{2(2)}.$$
From the induction on $m$, we get
\begin{equation*}
(\vep\ot\id)\Delta(w_{1})=\beta_{\ell}(w_{1})\,\text{ and }\, (\vep\ot\id)\Delta(w_{2})=\beta_{\ell}(w_{2}),
\end{equation*}
that is,
\begin{eqnarray}
\sum^{}_{(w_{1})}\vep(w_{1(1)})\ot w_{1(2)}=1_{\bfk}\ot w_{1}\text{ and }
\sum^{}_{(w_{2})}\vep(w_{2(1)})\ot w_{2(2)}=1_{\bfk}\ot w_{2},
\mlabel{eq:counit0}
\end{eqnarray}
which implies
\begin{align*}
(\vep\ot\id)\Delta(w)&=(\vep\ot\id)\Delta(w_{1}~\diamond~ w_{2})\\
&=(\vep\ot\id)(\Delta(w_{1})~\diamond~\Delta(w_{2}))
\quad(\text{by Eq.~(\mref{eq:Morphism})})\\
&=\sum_{(w_{1}),\, (w_{2})} \vep(w_{1(1)}~\diamond~ w_{2(1)})\ot(w_{1(2)}~\diamond~ w_{2(2)})\\
&=\sum_{(w_{1}),\,(w_{2})}
(\vep(w_{1(1)})\vep(w_{2(1)}))\ot(w_{1(2)}~\diamond~ w_{2(2)})\quad(\text{by Lemma~\mref{lem:counitprod})}\\
&=\left(\sum^{}_{(w_{1})}\vep(w_{1(1)})\ot w_{1(2)}\right)
\left(\sum^{}_{(w_{2})}\vep(w_{2(1)})\ot w_{2(2)}\right)\\
&=(1_{\bfk}\ot w_{1})(1_{\bfk}\ot w_{2})\quad(\text{by Eq.~(\mref{eq:counit0})})\\
&=1_{\bfk}\ot(w_{1}~\diamond~w_{2})=1_{\bfk}\ot w =\beta_{\ell}(w).
\end{align*}
This completes the induction on the width $\w(w)$ and hence the induction on the depth $\dep(w)$.
\end{proof}

Now we put the pieces together to state our main result of this section.
Define a linear map
$$u : \bfk\rightarrow\FNA,\,\ 1_\bfk \mapsto \bfone.$$

\begin{theorem}
Let $X$ be a set. Then the sextuple $(\FNA, \diamond, u, \Delta, \vep, \NA)$ is a left counital cocycle bialgebra.
\mlabel{thm:main}
\end{theorem}
\begin{proof}
By Lemma~\mref{lem:ncfree}, the quadruple $(\FNA,\diamond,u, \NA)$ is an operated algebra.
From Lemmas~\mref{lem:coasso} and~\mref{lem:counitt}, the triple $(\FNA, \Delta, \vep)$ is a left counital coalgebra.
Finally, $(\FNA, \diamond, u, \Delta, \vep,\NA)$ is a left counital cocycle bialgebra by Eq.~(\mref{eq:Tree}) and Lemmas~\mref{lem:calgh}, \mref{lem:counitprod}.
\end{proof}

\section{Left counital Hopf algebra structures on free Nijenhuis algebras}
\mlabel{sec:hopf}
This section is devoted to a left counital Hopf algebraic structure on a free Nijenhuis algebra $\FNA$.
All algebras considered in this section are assumed to be of characteristic zero.
The following concepts are from~\cite[Definitions~4.1 and~4.4]{ZG}.

\begin{defn}
A left counital bialgebra $(H, m ,u,\Delta,\vep)$ is called a {\bf graded left counital bialgebra} if there are {\bfk}-modules $H^{(n)}, n\geq0$, of $H$ such that
\begin{enumerate}
\item
$H=\bigoplus\limits^{\infty}_{n=0}H^{(n)}$;
\mlabel{it:It1}
\item
$H^{(p)}H^{(q)}\subseteq H^{(p+q)}$;
\mlabel{it:It2}
\item
$\Delta(H^{(n)})\subseteq  (H^{(0)} \ot H^{(n)}) \oplus (\bigoplus\limits^{}_{p+q=n\atop p>0,\,q>0}H^{(p)}\otimes H^{(q)})$ for $n\geq 0$.
\mlabel{it:It3}
\end{enumerate}
Further $H$ is called {\bf connected} if in addition $H^{(0)}= \im u~ (= {\bfk})$ and $\ker\vep=\bigoplus_{n\geq 1}H^{(n)}$.
\end{defn}

Let $A$ be an algebra and $C$ a left counital coalgebra.
Denote $R:=\Hom(C,A)$. For $f,g\in R$, the {\bf convolution product} of $f$ and $g$ can still be defined by
$$f\ast  g:=m_A(f\ot g)\Delta_C.$$

\begin{defn}
Let $H=(H,m,u,\Delta,\vep)$ be a left counital bialgebra. Let $e:=u\vep$.
\begin{enumerate}
\item
A linear map $S$ of $H$ is called a {\bf right antipode} for $H$ if
\begin{equation*}
 \id_H\ast S=e.
\end{equation*}
\item
A left counital bialgebra with a right antipode is called a {\bf left counital right antipode Hopf algebra} or simply a {\bf left counital Hopf algebra}.
\end{enumerate}
\end{defn}

In setting of connected graded left counital bialgebras, right antipodes come for free.

\begin{lemma}{\rm (\cite[Theorem~4.6]{ZG})}
A connected graded  left counital  bialgebra is a left counital Hopf algebra.
\mlabel{lem:conn}
\end{lemma}

There is no accident that $\FNA$ is a connected graded left counital bialgebra, as we will soon see.
For a $w\in \frakX_\infty$, define the {\bf degree} of $w$ to be
\begin{equation}
\deg(w):= \deg_X(w) + \deg_{\NA}(w),
\mlabel{eq:ddeg}
\end{equation}
where $\deg_{\NA}(w)$ (resp. $\deg_X(w)$) denotes the number of occurrences of $\NA$ (resp. $x\in X$) in $w$.
For example, $\deg( \lc x\rc) = 2$ and $\deg(x\lc x\rc) = 3$. Define
\begin{equation*}
\FN^{(n)}:=\bfk\{w\in \frak X_\infty\mid \deg(w)=n\}, \text{ where } n\geq0.
\end{equation*}
Then
\begin{equation}
\FNA=\bigoplus\limits^{\infty}_{n=0}\FN^{(n)}\,,\, \FN^{(0)}= \bfk\,\text{ and }\,\NA(\FN^{(n)})\subseteq \FN^{(n+1)}.
\mlabel{eq:subset}
\end{equation}
Here the inclusion follows from $\deg(\NA(w)) = \deg(w)+1$ for $w\in \frak X_\infty$.
We will show below that the grading as above is compatible with the product $\diamond$ and coproduct $\Delta$
on $\FNA$.

\begin{lemma}
Let $X$ be a set and $\FNA$ the free Nijenhuis algebra on $X$. Then
\begin{equation}
\FN^{(p)}~\diamond~ \FN^{(q)}\subseteq  \FN^{(p+q)} \,\text{ for all } p, q\geq 0.
\mlabel{eq:mgrad}
\end{equation}
\mlabel{lem:mgrad}
\end{lemma}

\begin{proof}
Let $w\in \FN^{(p)}$ and $w'\in \FN^{(q)}$ be basis elements in $\frakX_\infty$. By Proposition~\mref{pp:cdiam}, we may suppose
$$w =w_{1}~\diamond~ \cdots ~\diamond~w_{m}\,\text{ and }\, w'=w'_{1}~ \diamond~\cdots ~\diamond~w'_{m'},$$
where $\w(w)=m$, $\w(w')=m'$  and $m,m'\geq 0$.
We prove the result by induction on the sum $s:=p+q\geq 0$.
For the initial step of $s=0$, we have $p=q=0$ and so $w=w'=\bfone$ by Eq.~(\mref{eq:subset}). Hence
$$w~\diamond~w'= \bfone \in \FN^{(0)} = \FN^{(p+q)}.$$

For the induction step, assume the result is true for the case of $s\leq k$ for a $k\geq 0$, and consider the case of $s=k+1$.
Under this assumption, we reduce to prove the result by induction on the sum of widths $t:=m+m'$.
Since $s= k+1\geq 1$, we have $t = m+m'\geq 1$.
If $t=1$, then either $m=0$ and $m'=1$ or $m=1$ and $m'=0$. Without loss of generality, let $m=0$ and $m'=1$.
Then $w=\bfone$, $p=0$ and so
$$w~\diamond~ w' = w' \in \FN^{(q)} = \FN^{(p+q)}.$$

Assume the result is valid for the case of $s=k+1$ and $t\leq\ell$, and consider the case of $s=k+1$ and $t=\ell+1\geq 2$.
We have three cases to consider
\begin{enumerate}
\item $\w(w)\geq 2, \w(w')\geq 2$;
\item $\w(w)\geq 2, \w(w')=1$;
\item $\w(w)=1, \w(w')\geq 2$.
\end{enumerate}
Here we only supply the explicit proof of the first case, because other cases are easier and can be checked similarly.
Suppose
$$\w(w)\geq2 \,\text{ and }\, \w(w')\geq2,$$
and denote by
$$\widetilde{w} := w_{1}~\diamond~ \cdots ~\diamond~w_{m-1}\,\text{ and }\, \widetilde{w}': = w'_{2}~\diamond~\cdots ~\diamond~w'_{m'}.$$
Then
$$w =\widetilde{w}~\diamond~ w_m \,\text{ and }\,  w' = w'_1 ~\diamond~\widetilde{w}' $$
and
\begin{equation}
p = \deg(w)=\deg(\widetilde{w})+\deg(w_m)\,\text{ and }\,
q = \deg(w')=\deg(w'_1)+\deg(\widetilde{w}').
\mlabel{eq:c0}
\end{equation}
By the associativity of the product $\diamond$ and Definition~\mref{de:alt},
$$w~\diamond~w' =(\widetilde{w}~\diamond~ w_m)~\diamond~(w'_1~\diamond~ \widetilde{w}')
= \widetilde{w}~\diamond~ (w_m~\diamond~w'_1) ~\diamond~\widetilde{w}' = \widetilde{w}~(w_m~\diamond~w'_1)~\widetilde{w}',$$
which implies
\begin{align*}
w~\diamond~w'
&= \widetilde{w}~(w_m~\diamond~w'_1)~\widetilde{w}'\\
&\in  \widetilde{w}~\FN^{(\deg(w_m)+\deg(w'_1))}~\widetilde{w}' \quad (\text{by induction hypothesis on $t$})\\
&\subseteq \FN^{( \deg(\widetilde{w})+\deg(w_m)+\deg(w'_1)
+\deg(\widetilde{w}'))} \quad (\text{by Eq.~(\mref{eq:ddeg}})) \\
&= \FN^{(p+q)} \quad(\text{by Eq.~(\mref{eq:c0})}).
\end{align*}
This finishes the induction on $t$ and hence the induction on $s$.
\end{proof}

\begin{lemma}
Let $X$ be a set and $\FNA$ the free Nijenhuis algebra on $X$. Then
$$\Delta(\FN^{(n)})\subseteq (\FN^{(0)}\ot \FN^{(n)}) \oplus (\bigoplus\limits^{}_{p+q=n\atop p>0,\, q>0}\FN^{(p)}\ot \FN^{(q)}) \,\text{ for all }\, n\geq 0.$$
\mlabel{lem:cmgrad}
\end{lemma}

\begin{proof}
Let $w\in\frak X_{\infty}$ be a basis element of $\FNA$. We proceed to prove the result by induction on
$n=\deg(w)\geq 0$.
For the initial step of $n=\deg(w)=0$, we have $w=\bfone$ and so by Eq.~(\mref{eq:Init})
$$\Delta(\bfone) = \bfone\ot \bfone \in \FN^{(0)} \ot \FN^{(0)}.$$
For the induction step, assume the result is true for the case of $n\leq k$ for a $k\geq 0$, and consider
the case of $n=k+1$. Under this assumption, we reduce to prove the result
by induction on the width $m=\w( w)$. Since $n = \deg(w)\geq 1$, we have $w\neq \bfone$ and so
$m=\w(w)\geq 1$.

If $m=1$, then either $w = x$ for some $x\in X$ or $w= \NA(\lbar{ w})$ for some  $\lbar{ w}\in \frak X_\infty$. If $w=x$ for some $x\in X$,
then $n=\deg(w)=1$ and $$\Delta(w) = \bfone \ot x \in \FN^{(0)} \ot \FN^{(1)} = \FN^{(0)} \ot \FN^{(n)}.$$
If $ w=\NA(\lbar{ w})$ for some $\lbar{ w}\in \frak X_\infty$, then
$$\deg(\lbar{w}) = \deg(w)-1 =n-1= k.$$
Using the induction hypothesis on $n$, we get
$$\Delta(\lbar{ w})\in (\FN^{(0)}\ot \FN^{(k)}) \oplus (\bigoplus\limits^{}_{p+q=k\atop p>0,\, q>0}\FN^{(p)}\ot \FN^{(q)}).$$
By Eqs.~(\mref{eq:Tree}) and~(\mref{eq:subset}),
\begin{align*}
\Delta( w)=\Delta(\NA(\lbar{ w})) &= (\id\ot \NA)\Delta(\lbar{ w})\\
&\in \Big(\FN^{(0)}\ot \NA(\FN^{(k)}) \Big) \oplus \Big(\bigoplus\limits^{}_{p+q=k\atop p>0,\, q>0}\FN^{(p)}\ot \NA(\FN^{(q)})\Big) \\
& \subseteq   (\FN^{(0)}\ot \FN^{(k+1)}) \oplus (\bigoplus\limits^{}_{p+q=k+1\atop p>0,\, q>1}\FN^{(p)}\ot \FN^{(q)})\\
& \subseteq   (\FN^{(0)}\ot \FN^{(k+1)}) \oplus (\bigoplus\limits^{}_{p+q=k+1\atop p>0,\, q>0}\FN^{(p)}\ot \FN^{(q)}).
\end{align*}
Assume the result holds for the case of $n=k+1$ and $m\leq\ell$, and consider the case of $n=k+1$ and $m=\ell+1\geq 2$.
From Proposition~\mref{pp:cdiam}, we may write
$$ w=u~\diamond~ v,\,\text{ where }\, u, v\in \frak X_\infty\,\text{ and }\, 0< \w( u), \w(v) < \ell+1.$$
Using the induction hypothesis on $m$, we get
$$\Delta(u)\in (\FN^{(0)}\ot \FN^{(\deg(u))}) \oplus (\bigoplus\limits^{}_{p+q=\deg(u)\atop p>0,\,q>0}\FN^{(p)}\ot \FN^{(q)})$$
and
$$\Delta(v)\in  (\FN^{(0)}\ot \FN^{(\deg(v))}) \oplus ( \bigoplus\limits^{}_{p'+q'=\deg(v)\atop p'>0,\,q'>0}\FN^{(p')}\ot \FN^{(q')}).$$
Since $w = u~\diamond~v = uv$, we have
\begin{equation}
\deg( u)+\deg(v) = \deg( w) = n = k+1.
\mlabel{eq:deguv}
\end{equation}
Thus we obtain
\begin{align*}
\Delta( w)=&\ \Delta( u)~\diamond~ \Delta( v)\\
\in& \ \left( (\FN^{(0)}\ot \FN^{(\deg(u))}) \oplus (\bigoplus\limits^{}_{p+q=\deg(u)\atop p>0,\,q>0}\FN^{(p)}\ot \FN^{(q)})\right)~\diamond~
\left( (\FN^{(0)}\ot \FN^{(\deg(v))}) \oplus ( \bigoplus\limits^{}_{p'+q'=\deg(v)\atop p'>0,\,q'>0}\FN^{(p')}\ot \FN^{(q')})\right)\\
=&\ (\FN^{(0)}~\diamond~\FN^{(0)}) \ot (\FN^{(\deg(u))}~\diamond~\FN^{(\deg(v))}) \oplus
\left(\bigoplus\limits^{}_{p'+q'=\deg(v)\atop p'>0,\,q'>0}(\FN^{(0)}~\diamond~\FN^{(p')}) \ot (\FN^{(\deg(u))}~\diamond~\FN^{(q')}) \right)\\
&\oplus \left(\bigoplus\limits^{}_{p+q=\deg(u)\atop p>0,\,q>0}(\FN^{(p)}~\diamond~\FN^{(0)}) \ot (\FN^{(q)}~\diamond~\FN^{(\deg(v))}) \right)\\
&\oplus \left( \bigoplus\limits^{}_{p+q = \deg(u), \, p'+q'= \deg(v)\atop p>0,\,q>0, p'>0,\,q'>0}\left(\FN^{(p)}~\diamond~\FN^{(p')}\right)\ot \left(\FN^{(q)}~\diamond~\FN^{(q')}\right) \right) \\
%
%
\subseteq &\ (\FN^{(0)} \ot \FN^{(n)} ) \oplus  ( \bigoplus\limits^{}_{p+q= n \atop p>0, q>0}\FN^{(p)}\ot \FN^{(q)}) \,\quad\text{ (by Eqs.~(\mref{eq:mgrad}) and~(\mref{eq:deguv}))}.
\end{align*}
This completes the induction on $\w(w)$ and hence the induction on $\deg(w)$.
\end{proof}

Now we arrive at the main result of the paper.

\begin{theorem}
Let $X$ be a set and $\FNA$ the free Nijenhuis algebra on $X$. Then the quintuple $(\FNA, \diamond, \NA, \Delta, \varepsilon)$ is a connected graded left counital bialgebra and hence a left counital Hopf algebra.
\mlabel{thm:Hfree}
\end{theorem}
\begin{proof}
The $(\FNA, \diamond, \NA, \Delta, \varepsilon)$ is a left counital bialgebra by Theorem~\mref{thm:main}, and
further is graded by Eq.~(\mref{eq:subset}) and  Lemmas~\mref{lem:mgrad}, ~\mref{lem:cmgrad}.
The connectedness follows from Eqs.~(\mref{eq:vep}) and~(\mref{eq:subset}). Therefore the result is valid from Lemma~\mref{lem:conn}.
\end{proof}

\smallskip

\noindent {\bf Acknowledgements}: This work was supported by the National Natural Science Foundation of
China (Grant No.~11501466), Fundamental Research Funds for the Central
Universities (Grant No.~lzujbky-2017-162), and the Natural Science Foundation of Gansu Province (Grant
No.~17JR5\\RA175) and Shandong Province (No. ZR2016AM02).

We thank the anonymous referee for valuable suggestions helping to improve the paper.

\end{document}